\documentclass[reqno]{amsart}
\usepackage{amssymb,amsmath,amsthm,amscd,latexsym,amsfonts}
\usepackage{mathtools}
\usepackage[english]{babel}
\usepackage[T1]{fontenc}
\usepackage{inputenc}
\usepackage[arrow,matrix]{xy}
\usepackage{amsaddr}
\usepackage{graphicx}
\usepackage{cite}
\newtheorem{thm}{Theorem}[section]
\newtheorem{cor}[thm]{Corollary}
\newtheorem{lem}[thm]{Lemma}
\newtheorem{prop}[thm]{Proposition}

\theoremstyle{definition}
\newtheorem{defn}[thm]{Definition}

\theoremstyle{remark}
\newtheorem{rem}[thm]{Remark}
\numberwithin{equation}{section}
\DeclareMathOperator{\Der}{Der}

\DeclareMathOperator{\Orb}{Orb}
\DeclareMathOperator{\Ann}{Ann}
\DeclareMathOperator{\gr}{gr}
\DeclareMathOperator{\cent}{Center}
\DeclareMathOperator{\GL}{GL}
\DeclareMathOperator{\Leib}{\texttt{Leib}}
\begin{document}

\title[Solvable Leibniz algebras with naturally graded non-Lie $p$-filiform nilradicals]{Solvable Leibniz algebras with naturally graded non-Lie $p$-filiform nilradicals}
\author{J.~Q.~Adashev\textsuperscript{1}, M.~Ladra\textsuperscript{2}, B.~A.~Omirov\textsuperscript{1}}

\address{\textsuperscript{1}Institute of Mathematics, National University of Uzbekistan, 100125, Tashkent,  Uzbekistan, adashevjq@mail.ru, omirovb@mail.ru}
\address{\textsuperscript{2}Department of Algebra, University of Santiago de Compostela, 15782 Santiago de Compostela, Spain, manuel.ladra@usc.es}

\begin{abstract} In this paper solvable Leibniz algebras with naturally graded non-Lie $p$-filiform $(n-p\geq4)$ nilradical and with one-dimensional complemented space of nilradical are described. Moreover, solvable Leibniz algebras with abelian nilradical and extremal (minimal, maximal) dimensions of complemented space nilradical are studied. The rigidity of solvable Leibniz algebras with abelian nilradical and maximal dimension of its complemented space is proved.
\end{abstract}

\subjclass[2010] {17A32, 17A36, 17B30, 17B56.}

\keywords{Lie algebra, Leibniz algebra, solvable algebra, abelian algebra, nilradical, natural graduation, $p$-filiform algebra, group of cohomology, rigidity.}

\maketitle





\section{Introduction}

Leibniz algebras are generalizations of Lie algebras and they have been introduced by Loday in \cite{Loday}
as a non-antisymmetric version of Lie algebras.  These algebras preserve a unique property of Lie algebras - the right multiplication operators are derivations.
 Since the 1993 when Loday's work was published, many researchers have been attracted to Leibniz algebras, with remarkable activity during the last decades.

From the classical theory of Lie algebras it is well known that the study of finite-dimensional Lie algebras was reduced to the nilpotent ones \cite{Mal}.
In the Leibniz algebra case we have an analogue of Levi's theorem \cite{Bar}. Namely, the decomposition of a Leibniz algebra into a semidirect sum of its solvable radical and a semisimple Lie algebra is obtained.
 The semisimple part can be described from simple Lie ideals (see \cite{Bar}) and therefore, the main problem is to study the solvable radical.
 Based on the work of \cite{Mub}, a new approach for the investigation of solvable Lie algebras by using their nilradicals is developed in the works \cite{AnCaGa1,AnCaGa1x,AnCaGa2,NdWi,SnWi,TrWi,WaLiDe} and others.
  This approach is based on the nil-independent derivations of nilradical.
   In fact, for a given solvable Leibniz algebra with a fixed nilradical the complemented space to nilradical
    has a basis with the condition that the restriction of operators of right multiplication on a basis element is nil-independent derivation of nilradical \cite{Nulfilrad}.

Since the description of finite-dimensional solvable Lie algebras is a boundless problem, lately  geometric approaches are developing.
 Relevant tools of geometric approaches are Zariski topology and natural action of linear reductive group on varieties of algebras in a such way that orbits under the action consists of isomorphic algebras.
  It is a well-known result of algebraic geometry that any algebraic variety (evidently, algebras defined via identities form an algebraic variety) is a union of a finite number of irreducible components.
 The most important algebras are those whose orbits under the action are open sets in sense of Zariski topology.
 The algebras of a variety with open orbits are important since the closures of orbits of such algebras form irreducible components of the variety.
 At the same time there exist an irreducible component which is not closure of orbit of any algebra.
  This fact does not detract the importance of algebras with open orbits.
   This is a motivation of many works focused to discovering of algebras with open orbits and to description of sufficient properties of such algebras \cite{Burde,O'Halloran1,O'Halloran2}.

In this paper under the condition to a operator of right multiplication we classify solvable Leibniz algebras with abelian nilradical and with one-dimensional complemented space of nilradical.
 The classification of solvable Leibniz algebras with abelian nilradical and maximal dimension of complemented space is obtained, as well.
The rigidity of such algebras is proved. Moreover, we present a description of solvable algebras with naturally graded non-Lie $p$-filiform nilradicals ($n-p\geq 4$)
 under some condition to operator of right multiplication to an element of one-dimensional complemented space to nilradical. The classifications problems of obtained algebras of the descriptions are studied.

Throughout the paper we consider finite-dimensional vector spaces and algebras over the field $\mathbb{C}$.
 Moreover, in the multiplication table of an algebra omitted products are assumed to be zero and if it is not noticed we shall consider non-nilpotent solvable algebras.

\section{Preliminaries}

In this section we give necessary definitions and preliminary results.

\begin{defn}[\cite{Loday}] A  vector space with bilinear bracket $(L,[-,-])$ over a field $\mathbb{F}$ is called a \emph{Leibniz algebra} if for any $x,y,z\in L$ the so-called Leibniz identity
\[ \big[x,[y,z]\big]=\big[[x,y],z\big] - \big[[x,z],y\big] \]  holds, or equivalently, $\big[[x,y],z\big]=    \big[[x,z],y\big] + \big[x,[y,z]\big]$.
\end{defn}

Here, we adopt the right Leibniz identity; since the bracket is not skew-symmetric, there exists the version corresponding to the left Leibniz identity,
 \[ \big[[x,y],z\big] =  \big[x,[y,z]\big] - \big[y,[x,z]\big] \,. \]

For examples of Leibniz algebras we refer to papers \cite{Loday} and \cite{Loday1}.

Further we will use the notation
\[ {\mathcal L}(x, y, z)=[x,[y,z]] - [[x,y],z] + [[x,z],y].\]
It is obvious that Leibniz algebras are determined by the identity
${\mathcal L}(x, y, z)=0$.

From the Leibniz identity we conclude that the elements $[x,x], [x,y]+[y,x]$ for any $x, y \in L$ lie in  $\Ann_r(L) =\{x \in L \mid [y,x] = 0,\ \text{for \ all}\ y \in L \}$,
 the \emph{right annihilator} of the Leibniz algebra $L$. Moreover, it is easy to see that $\Ann_r(L)$ is a two-sided ideal of $L$.

The two-sided ideal  $\cent(L)=\{x\in L  \mid   [x,y]=0=[y,x], \ \text{for \ all}\ y \in L \}$ is said to be the \emph{center} of $L$.

\begin{defn} \label{der} A linear map $d \colon L \rightarrow L$ of a Leibniz algebra $(L,[-,-])$ is said to be a \emph{derivation} if for all $x, y \in L$ the following  condition holds:
\begin{equation}\label{eq0}
d([x,y])=[d(x),y] + [x, d(y)].
\end{equation}

\end{defn}

Note that the right multiplication operator $\mathcal{R}_x \colon L \to L, \mathcal{R}_x(y)=[y,x],  y \in L$, is a derivation (for a left Leibniz algebra $L$,
 the left multiplication operator $\mathcal{L}_x \colon L \to L, \mathcal{L}_x(y)=[x,y],  y \in L$, is also derivations).

\begin{defn} For a given Leibniz algebra $(L,[-,-])$ the sequences of two-sided ideals defined recursively as follows:
\[L^1=L, \ L^{k+1}=[L^k,L],  \ k \geq 1, \qquad \qquad
L^{[1]}=L, \ L^{[s+1]}=[L^{[s]},L^{[s]}], \ s \geq 1,
\]
are said to be the \emph{lower central} and the \emph{derived series} of $L$, respectively.
\end{defn}

\begin{defn} A Leibniz algebra $L$ is said to be
\emph{nilpotent} (respectively, \emph{solvable}), if there exists $n\in\mathbb N$ ($m\in\mathbb N$) such that $L^{n}=0$ (respectively, $L^{[m]}=0$).
 The minimal number $n$ (respectively, $m$) with such property is said to be the \emph{index of nilpotency} (respectively, \emph{index of solvability}) of the algebra $L$.
\end{defn}

Evidently, the index of nilpotency of an $n$-dimensional nilpotent algebra is not greater than $n+1$.

\begin{defn} The maximal nilpotent ideal of a Leibniz algebra is said to be the \emph{nilradical} of the  algebra.
\end{defn}

Let $R$ be a solvable Leibniz algebra with nilradical $N$. We denote by $Q$ the complementary vector space  of the nilradical $N$ to the algebra $R$.
 Let us consider the restrictions to $N$ of the right multiplication operator on an element $x \in Q$ (denoted by $\mathcal{R}_{{x |}_{N}}$).
  From \cite{Nulfilrad} we know that for any $x \in Q$, the operator $\mathcal{R}_{{x |}_{N}}$ is a non-nilpotent  derivation of $N$.

Let $\{x_1, \dots, x_m\}$ be a basis of $Q$, then for any scalars $\{\alpha_1, \dots, \alpha_m\}\in
\mathbb{C}\setminus\{0\}$, the matrix $\alpha_1\mathcal{R}_{{x_1 |}_{N}}+\dots+\alpha_m\mathcal{R}_{{x_m|}_{N}}$ is non-nilpotent,
which means that the elements $\{x_1, \dots, x_m\}$ are nil-independent \cite{Mub}. Therefore, we have that the dimension of $Q$
 is bounded by the maximal number of nil-independent derivations of the nilradical $N$ (see \cite[Theorem 3.2]{Nulfilrad}).
  Moreover, similar to the case of Lie algebras, for a solvable Leibniz algebra $R$ the inequality $\dim N \geq \frac{\dim R}{2}$ holds.

Let $L$ be a nilpotent Leibniz algebra and $x\in L \setminus L^2$. Denote by $C(x)=(n_1,n_2,
\dots,n_k)$ the decreasing sequence which consists of the dimensions of the Jordan blocks of the operator $R_x$. On the set of such sequences we consider the lexicographic order.

\begin{defn}  The sequence $C(L)=\max\limits_{x\in L \setminus L^2}C(x)$
is called the \emph{characteristic sequence} of the Leibniz algebra $L$.
\end{defn}

Below we define the notion $p$-filiform Leibniz algebra.

\begin{defn} A Leibniz algebra $L$ is called $p$-\emph{filiform}  if
$C(L)=(n-p,\underbrace{1,\dots,1}_{p})$, where $p\geq 0$.
\end{defn}

Note that above definition, when $p>0$ agrees with the definition of $p$-filiform Lie algebras \cite{Gomez}.
Since in the case of Lie algebras there is no singly-generated algebra, the notion of $0$-filiform algebra for Lie algebras has no sense,
 while for the Leibniz algebras case in each dimension there exists up to isomorphism a unique null-filiform algebra \cite{Omir1}.

\begin{defn}  Given an $n$-dimensional $p$-filiform Leibniz algebra $L$, put $L_i=L^i/L^{i+1},\ 1\leq i\leq n-p$, and $\gr L =L_1 \oplus L_2 \oplus \cdots \oplus L_{n-p}$.
 Then $[L_i,L_j]\subseteq L_{i+j}$ and we obtain the graded algebra $\gr L$. If $\gr L$ and $L$ are isomorphic, $\gr L\cong L$, we say that $L$ is \emph{naturally graded}.
\end{defn}

Due to voluminous of the list of naturally graded $p$-filiform ($1\leq p\leq n-1$) Lie algebras we refer the reader to the work \cite{Cabezas}.

Since we shall consider naturally graded non-Lie $p$-filiform nilradical we give its classification.

\begin{thm}[\cite{p-filLeibniz}] An arbitrary $n$-dimensional naturally graded non-split $p$-filiform Leibniz algebra $(n-p\geq 4)$ is isomorphic to one of the following non-isomorphic algebras:

$p=2k$ is even

\[\mu_1:\left\{\begin{array}{ll}
[e_i,e_1]=e_{i+1}, & 1\leq i\leq n-2k-1,\\[1mm]
[e_1, f_j] =f_{k+j}, & 1\leq j\leq k,\\[1mm]
 \end{array}\right.\]

\[\mu_2:\left\{\begin{array}{ll}
[e_i,e_1]=e_{i+1}, & 1\leq i\leq n-2k-1,\\[1mm]
[e_1,f_1]=e_{2}+f_{k+1}, &\\[1mm]
[e_i,f_1]=e_{i+1}, & 2\leq i\leq n-2k-1,\\[1mm]
[e_1, f_j] =f_{k+j}, & 2\leq j\leq k,\\[1mm]
 \end{array}\right.\]

$p=2k+1$ is odd

\[\mu_3:\left\{\begin{array}{ll}
[e_i,e_1]=e_{i+1}, & 1\leq i\leq n-2k-2,\\[1mm]
[e_1,f_j]=f_{k+1+j}, & 1\leq j\leq k,\\[1mm]
[e_i, f_{k+1}] =e_{i+1}, & 1\leq i\leq n-2k-2,\\[1mm]
 \end{array}\right.\]
where $\{e_1,e_2,\dots,e_{n-p},f_1,f_2,\dots,f_{p}\}$ is a basis
of the algebra.
\end{thm}

In order to simplify our further calculations for the algebra $\mu_3$, by taking the change of basis in the following form:
\[
e_1^\prime=e_1, \quad e_2^\prime=e_1-f_{k+1}, \quad e_{i+1}^\prime=e_{i}, \  2\leq i\leq n-2k-1,\quad f_{j}^\prime=f_{j}, \quad f_{k+j}^\prime=f_{k+1+j}, \ 1\leq j\leq k,
\]
we obtain the table of multiplication of the algebra $\mu_3$, which we shall use throughout the paper:
\[ \mu_3:\left\{\begin{array}{ll}
[e_1,e_1]=e_{3}, & \\[1mm]
[e_i,e_1]=e_{i+1}, & 2\leq i\leq n-2k-1,\\[1mm]
[e_1,f_j]=f_{k+j}, & 1\leq j\leq k,\\[1mm]
[e_2,f_j]=f_{k+j}, & 1\leq j\leq k.\\[1mm]
 \end{array}\right.\]
For acquaintance with the definition of cohomology group of Leibniz algebras and its applications
to the description of the variety of Leibniz algebras (similar to Lie algebras case) we refer the reader
to the papers \cite{Balavoine,Gersten,GozKhak,Loday,Loday1,Nijen}.
Here we just recall that the second cohomology group of a Leibniz algebra $L$ with coefficients in a corepresentation $M$
is the quotient space
\[HL^2(L, M) = ZL^2(L, M)/BL^2(L, M),\]
where the 2-cocycles $\varphi\in ZL^2(L, M)$ and the 2-coboundaries $f\in BL^2(L, M)$
are defined as follows

\begin{equation}\label{E.Z2}(d^2\varphi)(a,b,c)=[a,\varphi(b,c)]-[\varphi(a,b),c]+[\varphi(a,c),b]+\varphi(a,[b,c])-
\varphi([a,b],c)+\varphi([a,c],b)=0 \end{equation}
and

\begin{equation}\label{E.B2} f(a,b)=[d(a),b]+[a,d(b)]-d([a,b]) \ \mbox{for some linear map} \ d. \end{equation}

The linear reductive group $\GL_n(\mathbb{F})$ acts on $\Leib_n$, the variety of $n$-dimensional Leibniz algebra structures, via change of basis, i.e.,
\[(g*\lambda)(x,y)=g \Big(\lambda \big(g^{-1}(x),g^{-1}(y) \big) \Big), \quad  g \in \GL_n(\mathbb{F}), \  \lambda \in \Leib_n.\]

The orbits $\Orb(-)$ under this action are the isomorphism classes of algebras. Recall, Leibniz algebras with open orbits are called \emph{rigid}.
 Note that solvable Leibniz algebras of the same dimension also form an invariant subvariety of the variety of Leibniz algebras under the mentioned action.

\begin{rem} \label{rem1} Due to results of the paper [4] we have a sufficient condition for a Leibniz algebra being rigid algebra.
 Namely, if the second cohomology of a Leibniz algebra with coefficients in  itself is trivial, then it is a rigid algebra.
\end{rem}

\section{Solvable Leibniz algebras with abelian nilradical and extremal dimensions of complemented space $Q$}

We denote by ${\mathcal A(k)}$ the $k$-dimensional abelian algebras. For solvable Leibniz algebras with nilradical ${\mathcal A(k)}$
and dimension of complemented space of nilradical to an algebra is equal to $s$, we shall use
the  notation $R({\mathcal A(k)},s)$.

This section is devoted to the classification of solvable Leibniz algebras with nilradical ${\mathcal A(k)}$ under the condition that the complemented space to  the nilradical have extremal dimensions.
 The extremal dimensions of the complemented space means the minimal and maximal possible dimensions of the space. Evidently, the candidate for minimal dimension of complemented space is equal one.

Firstly, we consider the case of solvable algebras with one-dimensional complemented space of ${\mathcal A(k)}$.

Let $\{f_1,f_2,\dots,f_k,x\}$ be a basis of the algebra $R({\mathcal A(k)}, 1)$.

Evidently, the space of derivations of the algebra ${\mathcal A(k)}$ coincided with the space of $k \times k$ matrices.

Let us assume that the operator $\mathcal{R}_{{x |}_{{\mathcal A(k)}}}$ has Jordan block form, that is, $\mathcal{R}_{{x |}_{{\mathcal A(k)}}}=J_\lambda$.

\begin{thm} \label{thmQ=1} An arbitrary algebra of the family $R({\mathcal A(k)}, 1)$ is isomorphic to one of the following non-isomorphic algebras:
\[R_1: \quad \left\{\begin{array}{ll}
[f_i,x]=f_i+f_{i+1},&1\leq i\leq k-1,\\[1mm]
[f_k,x]=f_k,&\\[1mm]
\end{array}\right.
R_2: \quad \left\{\begin{array}{ll}
[f_i,x]=f_i+f_{i+1},&1\leq i\leq k-1,\\[1mm]
[f_k,x]=f_k,&\\[1mm]
[x,f_i]=-f_i-f_{i+1},&1\leq i\leq k-1,\\[1mm]
[x,f_k]=-f_k.&\\[1mm]
\end{array}\right.\]
\end{thm}
\begin{proof} Since $\lambda=0$ implies the nilpotency of the algebra $R$ we get $\lambda\neq0$.
Moreover, by scaling the basis elements
\[x^\prime=\lambda^{-1}x, \quad f_i^\prime=\lambda^{1-i}f_i, \ 1\leq i\leq k,\]
we can suppose $\lambda=1$.

Therefore, the  multiplication table of the algebra $R$ has the following form:
\[\left\{\begin{array}{ll}
[f_i,x]=f_i+f_{i+1},&1\leq i\leq k-1,\\[1mm]
[f_k,x]=f_k,&\\[1mm]
[x,f_i]=\sum\limits_{j=1}^{k}\alpha_{i,j}f_j,&1\leq i\leq k-1,\\[1mm]
[x,x]=\sum\limits_{i=1}^{k}\theta_{i}f_i.&\\[1mm]
\end{array}\right.\]

From the equalities ${\mathcal L}(x, f_i, x)={\mathcal L}(x, x, f_1)=0$ with $1\leq i\leq k$, derive the restrictions:
\begin{equation}\label{eq1}
\left\{\begin{array}{ll}
\alpha_{i,j}=0, & 1\leq j < i \leq k,\\[1mm]
\alpha_{i,j}=\alpha_{1,j-i+1},&1\leq i\leq j\leq k,\\[1mm]
\alpha_{1,1}^2=-\alpha_{1,1},&\\[1mm]
\sum\limits_{i=1}^{j}\alpha_{1,i}\alpha_{1,j-i+1}=-\alpha_{1,j-1}-\alpha_{1,j},& 2\leq j\leq k.\\[1mm]
\end{array}\right.
\end{equation}

Consider the possible cases.

{\bf Case 1.} Let $\alpha_{1,1}=0$. Then from the fourth equation of restrictions \eqref{eq1} we get $\alpha_{1,i}=0,\ \ 2\leq i\leq k$.

By taking the change in the following form:
\[ x^\prime=x+\sum\limits_{i=1}^{k}\sum\limits_{j=1}^{i}(-1)^{i+j-1}\theta_{j}f_{i}\]
we obtain the algebra $R_1$.

{\bf Case 2.} Let $\alpha_{1,1}=-1$. Then from restrictions \eqref{eq1} we obtain $\alpha_{1,2}=-1$ and $\alpha_{1,i}=0, \ 3\leq i\leq k$.
 The equality ${\mathcal L}(x, x, x)=0$ implies $\theta_{i}=0, \ 1\leq i\leq k$. Thus, we obtain the algebra $R_2$
\end{proof}

Now we shall classify the opposite case to one-dimensional complemented space $Q$, that is, we consider the maximal dimension of $Q$.

\begin{thm} The maximal possible dimension of algebras of the family $R({\mathcal A(k)}, s)$ is equal to $2k$, that is, $s=k$.
Moreover, an arbitrary algebra of the family $R({\mathcal A(k)}, k)$ is decomposed into a direct sum of copies of two-dimensional non-trivial solvable Leibniz algebras.
\end{thm}
\begin{proof} Let $\{x_1, x_2, \dots, x_s\}$ be a basis of the complemented space of ${\mathcal A(k)}$ to the algebra $R({\mathcal A(k)}, s)$.

Then from Leibniz identity
\[0=[f_i,[x_j,x_t]]=[[f_i,x_j],x_t]-[[f_i,x_t],x_j]\] we conclude that the operators $\mathcal{R}_{{x_i |}_{{\mathcal A(k)}}}, \ 1\leq i \leq s$, commute pairwise, that is,
$\mathcal{R}_{{x_j |}_{{\mathcal A(k)}}}\circ\mathcal{R}_{{x_t |}_{{\mathcal A(k)}}}=\mathcal{R}_{{x_t |}_{{\mathcal A(k)}}}\circ\mathcal{R}_{{x_j |}_{{\mathcal A(k)}}}$ for any $1\leq j, t\leq s$.
 This implies that all operators $\mathcal{R}_{{x_i |}_{{\mathcal A(k)}}}, \ 1\leq i \leq s$ could be simultaneously transformed to their Jordan forms by a basis transformation.

Let $\{\lambda_{1}^{(i)},\lambda_{2}^{(i)},\dots,\lambda_{k}^{(i)}\}$ be the eigenvalues of the operators
corresponding to $\mathcal{R}_{{x_i |}_{{\mathcal A(k)}}}, \ 1\leq i \leq s$.

Consider the vectors $\alpha_i=(\lambda_{1}^{(i)},\lambda_{2}^{(i)},\dots,\lambda_{k}^{(i)}), \ 1 \leq i\leq s$, of the $k$-dimensional vector space $\mathbb{C}^k$.
 Since $\Der({\mathcal A(k)})\cong M_{k,k}$ and nil-independent derivations are $\mathcal{R}_{{x_i |}_{{\mathcal A(k)}}}, \ 1\leq i \leq s$.
  We deduce that the maximal number of nil-independent among vectors $\alpha_i, \ 1\leq i \leq s$ is equal to $k$. This means that the maximal dimension of the complemented space is equal to $k$, i.e. $s=k$.

Without loss of generality we can assume that $\alpha_i=(0, \dots, 0, \underbrace{\lambda_i}_{i\text{-th place}}, 0, \dots,0)$,  \ $ \lambda_i\neq 0, \ 1 \leq i\leq k$, correspond to
$\mathcal{R}_{{x_i |}_{{\mathcal A(k)}}}, \ 1\leq i \leq k$.

Thus, we obtain the products in the algebras $R({\mathcal A(k)}, k)$:
\[\left\{\begin{array}{ll}
[f_i,x_j]=a_{i,j}f_{i+1},&1\leq i\leq k-1,\ 1\leq j\leq k,\ i\neq j,\\[1mm]
[f_i,x_i]=\lambda_{i}f_i+a_{i,i}f_{i+1},&1\leq i\leq k-1,\\[1mm]
[f_k,x_k]=\lambda_{k}f_k,&\\[1mm]
[f_{k},x_j]=0,& 1\leq j\leq k-1.\\[1mm]
\end{array}\right.\]

By scaling basis elements $x_i^\prime=\frac{1}{\lambda_i}, \ 1\leq i\leq k$, we obtain $\lambda_i=1$.

The equalities \[{\mathcal L}(f_i, x_i, x_j)={\mathcal L}(f_i, x_i, x_{i+1})=0\]
derive $a_{i,j}=0, \  1\leq i\leq k-1, \  1\leq j\leq k$.

Let us introduce the notations
\[[x_i,f_j]=\sum\limits_{t=1}^{k}\alpha_{i,j}^tf_t, \quad [x_i,x_j]=\sum\limits_{t=1}^{k}\gamma_{i,j}^tf_t, \quad 1\leq i,j\leq k.\]

From ${\mathcal L}(x_i, f_j, x_j)=0$ we have $[x_i,f_j]=\alpha_{i,j}f_j$.

Moreover, the equalities ${\mathcal L}(x_i, x_j, f_j)={\mathcal L}(x_i, x_j, f_i)=0$ imply the restrictions:
\[ \left\{\begin{array}{ll}
\alpha_{i,j}=0,&1\leq i,j\leq k, \ \ i\neq j,\\[1mm]
\alpha_{i,i}^2=-\alpha_{i,i},&1\leq i\leq k.\\[1mm]
\end{array}\right.\]


Consider the chain of equalities
\[ [x_i,[x_j,x_i]]=[[x_i,x_j],x_i]-[[x_i,x_i],x_j]=\gamma_{i,j}^if_i-\gamma_{i,i}^jf_j.\]

On the other hand, we have
\[ [x_i,[x_j,x_i]]=\Big[x_i,\sum\limits_{t=1}^{k}\gamma_{j,i}^tf_t\Big]=\alpha_{i,i}\gamma_{j,i}^if_i.\]

By comparing the coefficients at the basis elements we obtain
\begin{equation}\label{eq2}
\left\{\begin{array}{ll}
\alpha_{i,i}\gamma_{j,i}^i=\gamma_{i,j}^i,&1\leq i,j\leq k, \ \ i\neq j,\\[1mm]
\alpha_{i,i}\gamma_{i,i}^i=0,&1\leq i\leq k, \ \ i= j,\\[1mm]
\gamma_{i,i}^j=0,&1\leq i,j\leq k, \ \ i\neq j.\\[1mm]
\end{array}\right.
\end{equation}

We can assume that $\gamma_{i,i}^i=0$ for $1\leq i\leq k$. Indeed, if $\alpha_{i,i}\neq0$ for some $i$, then from the above restrictions we have $\gamma_{i,i}^i=0$.
 For those $i$ such that $\alpha_{i,i}=0$ by taking the change $x_i^\prime=x_i-\gamma_{i,i}^if_{i}$,  we again obtain $\gamma_{i,i}^i=0$.
  Therefore, $\gamma_{i,i}^i=0, \ 1\leq i\leq k$, that is, $[x_i,x_i]=0, \ 1\leq i\leq k$.

Consider the chain of equalities
\[ [x_i,[x_j,x_t]]=[[x_i,x_j],x_t]-[[x_i,x_t],x_j]=\gamma_{i,j}^tf_t-\gamma_{i,t}^jf_j.\]

On the other hand, we have
\[ [x_i,[x_j,x_t]]=\Big[x_i,\sum\limits_{l=1}^{k}\gamma_{j,t}^lf_l\Big]=\alpha_{i,i}\gamma_{j,t}^if_i.\]

From these we obtain \[ \gamma_{i,j}^t=0, \ 1\leq i, j, t \leq k, \ \ i\neq j, \ \ i\neq t, \ \ j\neq t.\]

By taking the change of basis element:
\[x_i^\prime=x_i-\sum\limits_{t=1,\ \ t\neq i}^{k}\gamma_{i,t}^tf_{t}, \ \ \ 1\leq i\leq k,\]
and by taking into account restrictions \eqref{eq2} we can conclude that $[x_i,x_j]=0,  \ \ 1\leq i,j\leq k$.

Thus, we have the  multiplication  table of the family of algebras $R({\mathcal A(k)}, k)$:
\[ \left\{\begin{array}{ll}
[f_i,x_i]=f_{i},&1\leq i\leq k,\\[1mm]
[x_i,f_i]=\alpha_{i,i}f_i,&1\leq i\leq k,\\[1mm]
\end{array}\right.\]
where $\alpha_{i,i}^2=-\alpha_{i,i}, \ 1\leq i\leq k$.
\end{proof}

\begin{rem} The number of non-isomorphic algebras in the family $R({\mathcal A(k)}, k)$ is equal to $k+1$.
\end{rem}
It should be noted that in the work \cite{Abror} the  algebras $R({\mathcal A(2)}, 2)$  were already classified.

Let $l_2: [e,x] = e$ and  $r_2: [e,x] = e, [x,e] = -e$ be two-dimensional non-Lie Leibniz and Lie algebras.

Consider the algebra ${\mathcal L_t}=l_2\oplus l_2\oplus \dots \oplus l_2 \oplus r_2 \oplus r_2  \oplus \dots  \oplus r_2$, where $t \ (0\leq t \leq k)$ is the number of entries of $l_2$ in ${\mathcal L_t}$. Then there exists a basis
$\{e_1, e_2, \dots, e_t, x_1, x_2, \dots, x_{t}, y_1, y_2, \dots, y_{k-t}, f_1, f_2, \dots f_{k-t}\}$  of ${\mathcal L_t}$ such that the  multiplication table has the form:
\[ [e_i,x_i] = e_i, \ \ 1 \leq i \leq t, \qquad [f_j,y_j] = -[y_j,f_j]=f_j, \ \ 1\leq j \leq k-t.\]

Let us present the general form of a derivation of the algebra ${\mathcal L_t}$.
\begin{prop}  Any derivation $d$ of the algebra from ${\mathcal L_t}$ has the following form:
\[ d(e_i)=a_ie_i, \quad 1\leq i\leq t, \qquad d(f_j)=b_jf_j, \qquad d(y_j)=c_jf_j, \quad 1\leq j\leq k-t.\]
\end{prop}
\begin{proof} The proof is carrying out by straightforward verification of derivation property \eqref{eq0}.
\end{proof}
Now we could to easily calculate the dimensions of the spaces $\Der({\mathcal L_t})$.
\begin{cor} \label{cor1}
\[ \dim \Der({\mathcal L_t}) = 2k-t, \qquad \dim BL^2({\mathcal L_t},{\mathcal L_t})=4k^2-2k+t.\]
\end{cor}

In order to prove the triviality of the second group of cohomology for the algebra ${\mathcal L_t}$ with coefficients in itself we need to describe the space of 2-cocycles.

\begin{prop} \label{prop111} Any element of $\varphi\in ZL^2({\mathcal L_t},{\mathcal L_t})$ has the following form:
\[\begin{array}{ll}
\varphi(x_i, x_j) = \alpha_{i,j} e_j, &  1\leq i,j\leq t,\\[1mm]
\varphi(e_i, x_i)
=\sum\limits_{m=1}^{t}\beta_{i,m}^1x_m+\sum\limits_{m=1}^{t}\beta_{i,m}^2e_m
+\sum\limits_{m=1}^{k-t}\beta_{i,m}^3y_m
+\sum\limits_{m=1}^{k-t}\beta_{i,m}^4f_m,&1 \leq i \leq t, ,\\[1mm]
\varphi(e_j, x_i) = -\beta_{i,j}^2e_i + \nu_{j,i}e_{j}, & 1\leq
i,j\leq t,\ i \neq j,\\[1mm]
\varphi(y_j, x_i)=\gamma_{j,i}^1e_i + \gamma_{j,i}^2f_j,&  1
\leq i \leq t, \ 1 \leq j \leq k-t,\\[1mm]
\varphi(y_j, y_i)=-\varphi(y_i, y_j) = \delta_{j,i}^1f_i +
\delta_{j,i}^2f_j,&  1
\leq i < j  \leq k-t, \\[1mm]
\varphi(f_i, y_i)
=\sum\limits_{m=1}^{t}\xi_{i,m}^1x_m+\sum\limits_{m=1}^{t}\xi_{i,m}^2e_m+\sum\limits_{m=1}^{k-t}
\xi_{i,m}^3y_m+\sum\limits_{m=1}^{k-t}\xi_{i,m}^4f_m ,&1 \leq i
\leq k-t,\\[1mm]
\varphi(f_j, y_i) =-\varphi(y_i, f_j) =-\xi_{i,j}^4f_i +
\eta_{j,i}f_{j},&
1 \leq i,j\leq k-t,\ i\neq j,\\[1mm]
\varphi(f_j, x_i) = -\xi_{i,j}^2e_i + \theta_{j,i}f_{j},& 1\leq i
\leq t, \ 1 \leq j \leq k-t,\\[1mm]
\varphi(e_i, y_j) = -\beta_{i,j}^4f_j + \tau_{j,i}e_{j},&  1 \leq
i \leq t, \ 1 \leq j \leq k-t,\\[1mm]
\varphi(x_i, y_j) = -\gamma_{i,j}^2f_j,&  1 \leq i \leq t, \ 1
\leq j \leq k-t,\\[1mm]
\varphi(e_i, e_j) = -\beta_{j,i}^1e_i , & 1 \leq i, j \leq
t,\\[1mm]
\varphi(f_j, e_i) = -\beta_{i,j}^3 f_j , & 1 \leq i \leq t, \ 1
\leq j \leq k-t,\\[1mm]
\varphi(y_j, e_i) = \beta_{j,i}^4 f_j ,  & 1 \leq i \leq t, \ 1
\leq j \leq k-t,\\[1mm]
\varphi(f_i, f_j) =-\varphi(f_j, f_i)= - \xi_{i,j}^3 f_i +
\xi_{j,i}^3 f_j,&  1 \leq
j<i \leq k-t,\\[1mm]
\varphi(e_i, f_j) = -\xi_{j,i}^1 e_i+\beta_{i,j}^3 f_j,&  1 \leq i
\leq t, \ 1 \leq j \leq k-t,\\[1mm]
\varphi(x_i, f_j) = - \theta_{j,i} f_j,&  1 \leq i \leq t, \ 1
\leq j \leq k-t.\end{array} \]
\end{prop}
\begin{proof} The proof is carrying out by straightforward calculations of equations \eqref{E.Z2} on the basis elements of the algebra ${\mathcal L_t}$.
\end{proof}

As consequence from Proposition \ref{prop111}  we have the following corollary.
\begin{cor} \label{cor2}
\[ \dim ZL^2({\mathcal L_t},{\mathcal L_t}) =4k^2-2k +t, \qquad \dim HL^2({\mathcal L_t},{\mathcal L_t}) =0.\]
\end{cor}

Now we give the main result regarding the rigidity of algebras ${\mathcal L_t}$.

\begin{thm} The algebra ${\mathcal L_t}$ is rigid algebra for any values of $t \ (0\leq t \leq k)$.
\end{thm}

\begin{proof} The proof of the theorem completes the argumentation of Remark \ref{rem1} and Corollary \ref{cor2}.
\end{proof}

\section{Solvable $n+1$-dimensional Leibniz algebras with $n$-dimensional naturally graded non-Lie $p$-filiform nilradicals.}

In this section we describe solvable Leibniz algebra with naturally graded non-Lie $p$-filiform nilradicals under the condition $\dim Q=1$.
 We focus in the non-split $p$-filiform non-Lie Leibniz algebras case.

\subsection{Derivations of algebras $\mu_i, i=1, 2, 3$}

\

In order to start the description we need to know the derivations of naturally graded non-Lie $p$-filiform Leibniz algebras.

\begin{prop} \label{prop1} Any derivation of the algebra $\mu_1$ has the following matrix form:
\[\mathbb{D}=\begin{pmatrix}
A&B\\
C&D
\end{pmatrix},\]
where
\[A=\sum_{i=1}^{n-2k}ia_{1}e_{i,i}+\sum_{i=1}^{n-2k-1}\sum_{j=i+1}^{n-2k}a_{j-i+1}e_{i,j}, \quad B=\sum_{i=1}^{2k}b_{i}e_{1,i}+\sum_{i=1}^{k}b_{i}e_{2,k+i},\]
\[C=\sum_{i=1}^{k}c_ie_{i,n-2k}, \quad D=\begin{pmatrix}
D_1&D_2\\
0&a_1\mathbb{E}+D_1
\end{pmatrix},\]
$\ A\in M_{n-2k,n-2k}, \ B\in M_{n-2k,2k}, \ C\in M_{2k,n-2k}, \ D_1, D_2, \mathbb{E} \in M_{k,k}$ and matrix units $e_{i,j}$.
\end{prop}
\begin{proof} Let $\{e_1,f_1,f_2,\dots,f_k\}$ be a generator basis elements of the algebra $\mu_1$.

We put
\[ d(e_1)=\sum\limits_{i=1}^{n-2k}a_ie_i+\sum\limits_{i=1}^{2k}b_if_i, \qquad d(f_i)=\sum\limits_{j=1}^{n-2k}c_{i,j}e_j+\sum\limits_{j=1}^{2k}d_{i,j}f_j, \quad 1\leq i\leq k.\]

From the derivation property \eqref{der} we have
\[
d(e_2)=d([e_1,e_1])=[d(e_1),e_1]+[e_1,d(e_1)]=2a_1e_2+\sum\limits_{i=3}^{n-2k}a_{i-1}e_{i}+
\sum\limits_{i=1}^{k}b_if_{k+i}.
\]

Buy applying the induction and the derivation property \eqref{der} we derive
\[ d(e_i)=ia_1e_i+\sum\limits_{t=i+1}^{n-2k}a_{t-i+1}e_t, \quad 3 \leq i \leq
n-2k.\]

Consider
\[ 0=d([f_i,e_1])=[d(f_i),e_1]+[f_i,d(e_1)]=\sum\limits_{j=1}^{n-2k-1}c_{i,j}e_{j+1}, \quad 1 \leq i \leq k.\]
Consequently,
\[ c_{i,j}=0,\quad  1 \leq i \leq k, \quad 1 \leq j \leq n-2k-1.\]

Similarly, from $d(f_{k+i})=d([e_1, f_i]), \ 1 \leq i \leq k$, we deduce
\[ d(f_{k+i})=a_{1}f_{k+i}+\sum\limits_{j=1}^{k}d_{i,j}f_{k+j}, \qquad 1\leq i\leq k.\]
\end{proof}

\begin{prop} \label{prop2} Any derivation of the algebra $\mu_2$ has the following matrix form:
\[\mathbb{D}=\begin{pmatrix}
A&B\\
C&D
\end{pmatrix},\]
where
\[A=\sum_{i=1}^{n-2k}(ia_{1}+(i-1)b_1)e_{i,i}+\sum_{i=1}^{n-2k-1}\sum_{j=i+1}^{n-2k}a_{j-i+1}e_{i,j},\]
 \[B=\sum_{i=1}^{2k}b_{i}e_{1,i}+\sum_{i=1}^{k}b_{i}e_{2,k+i}, \qquad C=\sum_{i=1}^{k}c_ie_{i,n-2k}, \qquad D=\begin{pmatrix}
D_1&D_2\\
0&D_3
\end{pmatrix},\]
\[D_1=\sum_{i=1}^{k}\sum_{j=2}^{k}d_{i,j}e_{i,j}+(a_1+b_{1})e_{1,1},\qquad D_3=D_1+a_1\mathbb{E}-\sum_{j=1}^{k}b_{j}e_{1,j},\]
with
$\ A\in M_{n-2k,n-2k}, \ B\in M_{n-2k,2k},  \ C\in M_{2k,n-2k}, \ D_1, D_2, D_3, \mathbb{E} \in M_{k,k}$
and matrix units $e_{i,j}$.
\end{prop}

\begin{proof} The proof is carrying out by straightforward calculation of the derivation property of the algebra $\mu_2$.
\end{proof}

\begin{prop} \label{prop3}  Any derivation of the algebra $\mu_3$ has
the following matrix form:
\[\mathbb{D}=\begin{pmatrix}
A&B\\
C&D
\end{pmatrix},\]
where
\[A=\sum_{i=1}^{n-2k}((i-1)a_{1}+a_2)e_{i,i}+\sum_{i=2}^{n-2k}a_{i}e_{1,i}+\sum_{i=3}^{n-2k-1}a_{i}e_{2,i}+\beta e_{2,n-2k}+\sum_{i=3}^{n-2k-1}\sum_{j=i+1}^{n-2k}a_{j-i+2}e_{i,j},\]
\[B=\sum_{i=1}^{2k}b_{1,i}e_{1,i}+\sum_{i=1}^{k}b_{2,i}e_{2,k+i}+\sum_{i=1}^{k}b_{1,i}e_{3,k+i}, \quad C=\sum_{i=1}^{k}c_ie_{i,n-2k}, \quad D=\begin{pmatrix}
D_1&D_2\\
0&(a_1+a_2)\mathbb{E}+D_1
\end{pmatrix}\]
with $A\in M_{n-2k,n-2k}, \ B\in M_{n-2k,2k}, \ C\in M_{2k,n-2k}, \ D_1,D_2, \mathbb{E} \in M_{k,k}$ and
matrix units $e_{i,j}$.
\end{prop}
\begin{proof}
The proof is carrying out by straightforward calculation of the derivation property of the algebra $\mu_3$.
\end{proof}

\subsection{Descriptions of algebras $R(\mu_i, 1), \ i=1, 2, 3$}

\

\

Let us consider the solvable algebra $R(\mu_i, 1)=\mu_i\oplus Q, \ i=1,2,3$, and a basis $\{e_1, e_2, \dots, e_{n-p}, f_1, f_2, \dots, f_p\}$.

We set ${\mathcal F} \coloneqq \{f_1,f_2,\dots,f_k\}$ and consider the projection of the operator $\mathcal{R}_{{x |}_{{\mathcal F}}}$ to the space ${\mathcal F}$ (denoted $\varepsilon(\mathcal{R}_{{x |}_{{\mathcal F}}}))$.
Let us suppose that there exists a basis of the space ${\mathcal F}$ such that the Jordan form of the operator
$\varepsilon(\mathcal{R}_{{x |}_{{\mathcal F}}})$ can be transformed into a Jordan block $J_{\lambda}$ with $\lambda\neq0$.

\begin{thm} \label{thmmu1} An arbitrary algebra of the family $R(\mu_1, 1)$ admits a basis such that its multiplication table  has the following form:
\[R(\mu_1, 1)(a_2, \dots, a_{n-2k+1}): \quad \left\{\begin{array}{ll}
[e_i,e_1]=e_{i+1},&1\leq i\leq n-2k-1,\\[1mm]
[e_1,f_i]=f_{k+i},&1\leq i\leq k,\\[1mm]
[e_i,x]=\sum\limits_{j=i+1}^{n-2k}a_{j-i+1}e_j,&1\leq i\leq n-2k,\\[1mm]
[f_i,x]=f_i+f_{i+1},&1\leq i\leq k-1,\\[1mm]
[f_k,x]=f_k,&\\[1mm]
[f_{i},x]=f_i+f_{i+1},&k+1\leq i\leq 2k-1,\\[1mm]
[f_{2k},x]=f_{2k},&\\[1mm]
[x,f_i]=-f_i-f_{i+1},&1\leq i\leq k-1,\\[1mm]
[x,f_k]=-f_k,&\\[1mm]
[x,x]=a_{n-2k+1}e_{n-2k}.&\\[1mm]
\end{array}\right.\]
\end{thm}
\begin{proof} From Proposition \ref{prop1} we have the products in the algebra $R(\mu_1, 1)$:
\[\left\{\begin{array}{ll}
[e_i,e_1]=e_{i+1},&1\leq i\leq n-2k-1,\\[1mm]
[e_1,f_i]=f_{k+i},&1\leq i\leq k,\\[1mm]
[e_1,x]=\sum\limits_{i=1}^{n-2k}a_ie_i+\sum\limits_{i=1}^{2k}b_if_i,&\\[1mm]
[e_2,x]=2a_1e_2+\sum\limits_{i=3}^{n-2k}a_{i-1}e_i+\sum\limits_{i=1}^{k}b_if_{k+i},&\\[1mm]
[e_i,x]=ia_1e_i+\sum\limits_{j=i+1}^{n-2k}a_{j-i+1}e_j,&3\leq i\leq n-2k,\\[1mm]
[f_i,x]=c_ie_{n-2k}+\sum\limits_{j=1}^{2k}d_{i,j}f_j,&1\leq i\leq k,\\[1mm]
[f_{k+i},x]=a_1f_{k+i}+\sum\limits_{j=1}^{k}d_{i,j}f_{k+j},&1\leq i\leq k.\\[1mm]
\end{array}\right.\]

Let us introduce the notations:
\[[x,e_1]=\sum\limits_{i=1}^{n-2k}\beta_ie_i+\sum\limits_{i=1}^{2k}\beta_{n-2k+i}f_i, \ [x,f_i]=\sum\limits_{j=1}^{n-2k}\gamma_{i,j}e_j+\sum\limits_{j=1}^{2k}\varphi_{i,j}f_j, \ 1\leq i\leq k, \ [x,x]=\sum\limits_{i=1}^{n-2k}\delta_{i}e_i+\sum\limits_{i=1}^{2k}\theta_{i}f_i.\]

Since the space ${\mathcal F}$ forms an abelian algebra, we are in the conditions of Theorem \ref{thmQ=1}.
 Moreover, the products $[e_1,e_1]=e_2, \ [e_1,f_i]=f_{k+i}$ ensure that $e_1, f_i\notin \Ann_r(R(\mu_1, 1), \ 1\leq i \leq k$.
  In particular, we are in the conditions of the algebra $R_2$. This implies the existence of a basis of ${\mathcal F}$ such that
\[\begin{array}{ll}
[f_i,x]=c_ie_{n-2k}+f_i+f_{i+1}+\sum\limits_{j=k+1}^{2k}d_{i,j}f_j,&1\leq i\leq k-1,\\[1mm]
[f_k,x]=c_ke_{n-2k}+f_k+\sum\limits_{j=k+1}^{2k}d_{k,j}f_j,&\\[1mm]
[x,f_i]=\sum\limits_{j=1}^{n-2k}\gamma_{i,j}e_j-f_i-f_{i+1}+\sum\limits_{j=k+1}^{2k}\varphi_{i,j}f_j,&1\leq i\leq k-1,\\[1mm]
[x,f_k]=\sum\limits_{j=1}^{n-2k}\gamma_{k,j}e_j-f_k+\sum\limits_{j=k+1}^{2k}\varphi_{k,j}f_j,&\\[1mm]
[x,x]=\sum\limits_{i=1}^{n-2k}\delta_{i}e_i+\sum\limits_{i=k+1}^{2k}\theta_{i}f_i.&\\[1mm]
\end{array}\]

The equalities \[{\mathcal L}(x, f_i, e_1)={\mathcal L}(e_1, x, e_1)={\mathcal L}(x, e_1, f_i)=0\]
imply
\[\begin{array}{ll}
\gamma_{i,j}=\beta_1=0, & 1\leq i\leq k, \ 1\leq j\leq n-2k-1,\\[1mm]
\alpha_1=0, \ \beta_{n-2k+i}=-b_i,  & 1\leq i\leq k,\\[1mm]
[x,f_{k+i}]=0,  & 1 \leq i\leq k.
\end{array}\]

Since $\{e_2,e_3,\dots,e_{n-2k},f_{k+1},\dots,f_{2k}\}\subseteq \Ann_r(R(\mu_1, 1))$ and $[x,x]\in \Ann_r(R(\mu_1, 1))$ while $e_1\notin \Ann_r(R(\mu_1, 1))$, we conclude $\delta_1=0$.

By setting
\[f_i^\prime=f_i-\chi_{i,n-2k}e_{n-2k}-\sum\limits_{j=k+1}^{2k}\psi_{i,j}f_j, \ 1\leq i \le k,\]
where $\chi_{k,n-2k}=\gamma_{k,n-2k}, \quad \psi_{k,j}=\varphi_{k,j}$
and parameters $\chi_{i,n-2k}, \psi_{i,j}$ for $1\leq i \leq k-1$, can be recursively obtained from the products
\[[x,f_i^\prime]+f_i^\prime=-f_{i+1}^\prime,\]
we can assume that
\[[x,f_i]=-f_i-f_{i+1}, \ \ 1\leq i\leq k-1, \quad [x,f_k]=-f_k.\]

The equalities ${\mathcal L}(x, f_i, x)=0, \ 1\leq i \leq k$ we derive $c_i=d_{i,j}=0, \ 1\leq i\leq k, \ \ k+1\leq
j\leq 2k$, that is, we obtain
\[[f_i,x]=f_i+f_{i+1}, \  1\leq i\leq k-1, \quad [f_k,x]=f_k.\]

By taking the change of basis element:
\[x^\prime=x-\sum\limits_{i=2}^{n-2k}\beta_ie_{i-1},\]
we obtain $[x^\prime,e_1]=-\sum\limits_{i=1}^{k}b_{i}f_i+\sum\limits_{i=1}^{k}\beta_{n-k+i}f_{k+i}$,
that is, we can assume that $\beta_{i}=0, \ 2\leq i\leq n-2k$.

The equality ${\mathcal L}(x, e_1, x)=0$ implies
\[\delta_{i}=0, \quad 2\leq i\leq n-2k-1, \qquad \beta_{n-k+i}=0, \quad 1\leq i\leq k.\]

By putting
\begin{align*}
e_1^\prime & =e_1+\sum\limits_{i=1}^{k}\sum\limits_{j=1}^{i}(-1)^{i+j-1}b_{j}f_{i}+
\sum\limits_{i=1}^{k}\sum\limits_{j=1}^{i}(-1)^{i+j-1}b_{k+j}f_{k+i},\\
x^\prime & =x+\sum\limits_{i=1}^{k}\sum\limits_{j=1}^{i}(-1)^{i+j-1}\theta_{k+j}f_{k+i},
\end{align*}
we deduce
\[b_{i}=b_{k+i}=\theta_{k+i}=0, \qquad 1\leq i\leq k.\]
Thus, we obtain the  multiplication tables of the algebras $R(\mu_1)(a_2, \dots, a_{n-2k+1})$ of the assertion of the theorem.
\end{proof}

In the next proposition a necessary and sufficient condition for the existence of an isomorphism between two algebras of the family $R(\mu_1, 1)(a_2, \dots, a_{n-2k+1})$ is established.

\begin{prop} \label{propmu1} Two algebras $R(\mu_1, 1)^\prime (a_2^\prime, \dots, a_{n-2k+1}^\prime)$ and $R(\mu_1, 1)(a_2, \dots, a_{n-2k+1})$ are isomorphic if and only if there exists $A\in \mathbb{C}$ such that
\[a_i^\prime=\frac{a_i}{A^{i-1}}, \quad 2\leq i\leq n-2k+1.\]
\end{prop}
\begin{proof}
Let us consider the general change of generator basis elements of the algebra $R(\mu_1, 1)$:
\[\begin{array}{ll}
e_1^\prime=\sum\limits_{i=1}^{n-2k}A_{i}e_i+\sum\limits_{i=1}^{2k}B_{i}f_i, &\\[1mm]
f_i^\prime=\sum\limits_{j=1}^{n-2k}C_{i,j}e_j+\sum\limits_{j=1}^{2k}D_{i,j}f_j,& 1\leq i\leq k,\\[1mm]
x^\prime=Hx+\sum\limits_{i=1}^{n-2k}E_{i}e_i+\sum\limits_{i=1}^{2k}F_{i}f_i.\\[1mm]
\end{array}
\]

From the products
\[[e_i^\prime,e_1^\prime]=e_{i+1}^\prime, \quad 1\leq i\leq n-2k-1,\] of the algebra $R(\mu_1, 1)^\prime$ we derive
\[e_{2}^\prime=A_1\sum\limits_{i=2}^{n-2k}A_{i-1}e_i+A_1\sum\limits_{i=1}^{k}B_{i}f_{k+i}, \qquad e_{i}^\prime=A_1^{i-1}\sum\limits_{j=i}^{n-2k}A_{j-i+1}e_j, \ 3\leq i\leq n-2k.\]

By considering
\[[x^\prime,e_1^\prime]=[f_i^\prime,e_1^\prime]=[f_i^\prime,f_i^\prime]=0, \quad 1\leq
i\leq k,\] we deduce
\[C_{i,j}=B_{i}=E_{j}=0,\qquad 1\leq i\leq k,\quad 1\leq j\leq n-2k-1.\]

Similarly, from the products
$[x^\prime,f_i^\prime]=-f_i^\prime-f_{i+1}^\prime, \ 1\leq i\leq k-1$, and $[x^\prime,f_k^\prime]=-f_k^\prime$
we conclude
\[\begin{array}{ll}
H=1, \ D_{i,j}=0,  & 1\leq j<i\leq k,\\[1mm]
D_{i,j}=D_{1,j-i+1}, & 1\leq i\leq j\leq k,\\[1mm]
C_{i,n-2k}=D_{i,k+j}=0, & 1\leq i\leq k, \ \ 1\leq j\leq k.\\[1mm]
\end{array}\]

The relations between parameters  $\{a_i^\prime\}$ and $\{a_i\}$ follow from the
products
\[[e_1^\prime,x^\prime]=\sum\limits_{i=2}^{n-2k}a_{i}^\prime e_i^\prime, \qquad [x^\prime,x^\prime]=a_{n-2k+1}^\prime e_{n-2k}^\prime.\]

Namely, we have
\[a_i^\prime=\frac{a_i}{A_1^{i-1}}, \quad 2\leq i\leq n-2k+1.\]
\end{proof}

Below, we present an analogue of Theorem \ref{thmmu1} for the family $R(\mu_2 , 1)$.

\begin{thm} An arbitrary algebra of the family $R(\mu_2 , 1)$ admits a basis such that its  multiplication table has the following form:
\[R(\mu_2, 1)(\alpha, \beta, \gamma): \quad\left\{\begin{array}{ll}
\mu_2,\\[1mm]
[e_1,x]=f_1+\alpha f_{k+1},&\\[1mm]
[e_2,x]=e_2+f_{k+1},&\\[1mm]
[e_i,x]=(i-1)e_i,&3\leq i\leq n-2k,\\[1mm]
[f_i,x]=-[x,f_i]=f_i+f_{i+1},&1\leq i\leq k-1,\\[1mm]
[f_k,x]=-[x,f_k]=f_k,&\\[1mm]
[f_{k+1},x]=f_{k+2},&\\[1mm]
[f_{i},x]=f_i+f_{i+1},&k+2\leq i\leq 2k-1,\\[1mm]
[f_{2k},x]=f_{2k},&\\[1mm]
[x,e_1]=-f_1+\beta\sum\limits_{i=1}^{k}(-1)^{i-1}f_{k+i},&\\[1mm]
[x,x]=\gamma f_{k+1}.&\\[1mm]
\end{array}\right.\]
\end{thm}

The classification of the family of algebras $R(\mu_2, 1)(\alpha, \beta, \gamma)$ is presented in the following proposition.

\begin{prop} Any algebra of the family $R(\mu_2, 1)(\alpha, \beta, \gamma)$ is isomorphic to one of the following pairwise non-isomorphic algebras:
\[R(\mu_2, 1)(0, 0, 0), \quad R(\mu_2, 1)(0, 0, 1), \quad R(\mu_2, 1)(0, 1, \gamma), \quad R(\mu_2, 1)(1, \beta, \gamma), \  \beta, \gamma \in \mathbb{C}.\]
\end{prop}
\begin{proof} In a similar way as in the proof of Proposition \ref{propmu1}, we consider the general transformation of generator basis elements:
\[\begin{array}{ll}
e_1^\prime=\sum\limits_{i=1}^{n-2k}A_{i}e_i+\sum\limits_{i=1}^{2k}B_{i}f_i,& \\[1mm]
f_i^\prime=\sum\limits_{j=1}^{n-2k}C_{i,j}e_j+\sum\limits_{j=1}^{2k}D_{i,j}f_j,& 1\leq i\leq k,\\[1mm]
x^\prime=\sum\limits_{i=1}^{n-2k}E_{i}e_i+\sum\limits_{i=1}^{2k}F_{i}f_i+Hx.
\end{array}
\]

Then from the products
\[[e_i^\prime,e_1^\prime]=e_{i+1}^\prime, \ 1\leq i\leq n-2k-1, \quad f_{k+1}^\prime=[e_1^\prime,f_1^\prime]-e_2^\prime, \quad f_{k+i}^\prime=[e_1^\prime,f_i^\prime], \ \ 2\leq i\leq k,\]
we derive the rest of new basis elements
\[\begin{array}{ll}
e_{2}^\prime=(A_1+B_1)\sum\limits_{i=2}^{n-2k}A_{i-1}e_i+A_1\sum\limits_{i=1}^{k}B_{i}f_{k+i},\\[1mm]
e_{i}^\prime=(A_1+B_1)^{i-1}\sum\limits_{j=i}^{n-2k}A_{j-i+1}e_j, & 3\leq i\leq n-2k, \\[1mm]
f_{k+1}^\prime=A_1\sum\limits_{i=1}^{k}(D_{1,i}-B_i))f_{k+i},& A_1\neq0\\[1mm]
f_{k+i}^\prime=A_1\sum\limits_{j=2}^{k}D_{i,j}f_{k+j}, & 2\leq i\leq k.\\[1mm]
\end{array}\]

By verifying the multiplications of $R(\mu_2, 1)(\alpha^\prime, \beta^\prime, \gamma^\prime)$ in the new basis
we obtain the relations between the parameters  $\{\alpha^\prime, \beta^\prime, \gamma^\prime\}$ and $\{\alpha, \beta, \gamma\}:$
\[\alpha^\prime=\frac{\alpha}{A_1}, \qquad \beta^\prime=\frac{\beta}{A_1}, \qquad \gamma^\prime=\frac{\gamma}{A_1^2}.\]
\end{proof}

Analogously, we have the description of solvable algebras.
\begin{thm} An arbitrary algebra of the family $R(\mu_3, 1)$ admits a basis such that its   multiplication table  has the following form:
\[R(\mu_3,1)(I):\quad \left\{\begin{array}{ll}
[e_i,e_1]=e_{i+1},&2\leq i\leq n-2k-1,\\[1mm]
[e_2,f_i]=f_{k+i},&1\leq i\leq k,\\[1mm]
[e_1,x]=a_1e_1+a_{n-2k}e_{n-2k}+b_{1}f_1+b_{2}f_{k+1},&\\[1mm]
[e_2,x]=(a_1+a_2)e_2+\sum\limits_{i=4}^{n-2k-1}a_{i}e_i+\beta e_{n-2k},&\\[1mm]
[e_3,x]=(2a_1+a_2)e_3+\sum\limits_{i=5}^{n-2k}a_{i-1}e_i+b_{1}f_{k+1},&\\[1mm]
[e_i,x]=((i-1)a_1+a_2)e_i+\sum\limits_{j=i+2}^{n-2k}a_{j-i+2}e_j,&4\leq i\leq n-2k,\\[1mm]
[f_i,x]=-[x,f_i]=f_i+f_{i+1},&1\leq i\leq k-1,\\[1mm]
[f_k,x]=-[x,f_k]=f_k,&\\[1mm]
[f_{k+i},x]=(a_1+a_2+1)f_{k+i}+f_{k+i+1},&1\leq i\leq k-1,\\[1mm]
[f_{2k},x]=(a_1+a_2+1)f_{2k},&\\[1mm]
[x,e_1]=-a_{1}e_1-b_{1}f_1+\sum\limits_{i=1}^{k}\theta_{n-k+i}f_{k+i},&\\[1mm]
[x,e_2]=\gamma e_{n-2k},&\\[1mm]
[x,x]=\delta_{1}e_{n-2k-1}+\delta_{2}e_{n-2k}+\delta_{3}f_{k+1},&\\[1mm]
\end{array}\right.\]
where $I=\{a_i, b_1, b_2, \beta, \gamma, \delta_{1}, \delta_{2}, \delta_{3}, \theta_{n-k+j}\}$ with $1\leq i \leq n-2k, \ i\neq 3, \ 1\leq j \leq k$,
 and for these parameters the following equalities hold true
\begin{equation} \label{eq11} \begin{array}{ll}
(a_1+a_2)\gamma=0, &(n-2k-2)a_1\gamma=0,\\[1mm]
\delta_{1}=-a_1a_{n-2k},& a_1b_{2}=(-1)^{k-1}(a_2+1)^k\theta_{n},\\[1mm]
\theta_{n-i}=(-1)^{i}(a_2+1)^i\theta_{n}, & 0\leq i\leq k-1.\end{array}
\end{equation}
\end{thm}

We set $J \coloneqq I\setminus \{a_i, b_1, b_2, \beta, \gamma, \delta_{1}, \delta_{2}, \delta_{3}, \theta_{n-k+j}\}$ with $4\leq i \leq n-2k, \ 1\leq j \leq k$.

\begin{lem} \label{lemmu3} Let algebras $R(\mu_3, 1)(a_1,a_2, J)$ and $R(\mu_3, 1)(a_1', a_2', J')$ are isomorphic. Then
\[a_1^\prime=a_1, \quad  a_2^\prime=a_2.\]
\end{lem}

\begin{proof} Let us take the general transformation of generators basis elements of the algebra $R(\mu_3)(a_1, a_2, J)$ in the following form:
\[\begin{array}{ll}
e_1^\prime=\sum\limits_{i=1}^{n-2k}A_{1,i}e_i+\sum\limits_{i=1}^{2k}B_{1,i}f_i,&
e_2^\prime=\sum\limits_{i=1}^{n-2k}A_{2,i}e_i+\sum\limits_{i=1}^{2k}B_{2,i}f_i,\\[1mm]
f_i^\prime=\sum\limits_{j=1}^{n-2k}C_{i,j}e_j+\sum\limits_{j=1}^{2k}D_{i,j}f_j,
\ 1\leq i\leq k, &
x^\prime=Hx+\sum\limits_{i=1}^{n-2k}E_{i}e_i+\sum\limits_{i=1}^{2k}F_{i}f_i.\end{array}\]

Then from the products $[e_i^\prime,e_1^\prime]=e_{i+1}^\prime, \ \ 2\leq i\leq
n-2k-1$, we obtain
\begin{align*}
e_{3}^\prime & =A_{1,1}\sum\limits_{i=3}^{n-2k}A_{2,i-1}e_i+A_{2,2}\sum\limits_{i=1}^{k}B_{1,i}f_{k+i},\\
e_{i}^\prime & =A_{1,1}^{i-2}\sum\limits_{j=i}^{n-2k}A_{2,j-i+2}e_j,\ \ 4\leq i\leq n-2k.
\end{align*}

Moreover, the equalities
\[[e_1^\prime,e_1^\prime]=[e_2^\prime,e_2^\prime]=[f_i^\prime,e_1^\prime]=[e_3^\prime,f_i^\prime]=0, \ 1\leq i\leq k,\]
imply
\[\begin{array}{ll}
A_{1,i}=0,& 2\leq i\leq n-2k-1,\\[1mm]
A_{2,1}=0=B_{2,i}=C_{i,j}=0,& 1\leq i\leq k, \ \ 1\leq j\leq n-2k-1.
\end{array}\]

The rest basis elements $f_{k+i}^\prime, \ 1\leq i\leq k$, can be obtained from the products $[e_2^\prime,f_i^\prime]=f_{k+i}^\prime, \ \ 1\leq i\leq k$. Namely, we have
\[f_{k+i}^\prime=A_{2,2}\sum\limits_{j=1}^{k}D_{i,j}f_{k+j},\ \ 1\leq i\leq k.\]

Consider $[x^\prime,f_k^\prime]=-f_k^\prime$. Then we derive
\begin{align*}
(D_{k,i}+D_{k,i+1})H & =D_{k,i+1}, && 1\leq i\leq k-1,\\
D_{k,1}H & =D_{k,1},\\
D_{k,i}E_2 & =-D_{k,k+i}, && 1\leq i\leq k,\\
C_{k,n-2k} & =0.
\end{align*}

If $H\neq1$, then $D_{k,i}=0, \  1\leq i\leq k$, which implies $f_{2k}^\prime=0$, this is a contradiction. Therefore, $H=1$ and
\[D_{k,i}=D_{k,k+i}=0,  \qquad 1\leq i\leq k-1.\]

Now from $[x^\prime,f_i^\prime]=-f_i^\prime-f_{i+1}^\prime$ with $1\leq i\leq k-1$, we get restrictions:
\begin{align*}
D_{i,j} & =0,  && 1\leq j<i\leq k,\\
D_{i,j} & =D_{1,j-i+1},  &&  1\leq i\leq j\leq k,\\
C_{i,n-2k} & =0,  &&  1\leq i\leq k-1,\\
D_{i,j}E_2 & =-D_{i,k+j}-D_{i+1,k+j}, &&  1\leq i\leq k-1, \ \ 1\leq j\leq k.
\end{align*}

By considering the products $[e_1^\prime,x^\prime]$ and $[e_2^\prime,x^\prime]$, we deduce
\[a_1^\prime=a_1, \qquad a_{2}^\prime=a_{2}.\]

\end{proof}

Let us introduce the notations
\[\begin{array}{ll}
I_1=\{a_i, \beta, \gamma, \delta_2\},& 4 \leq i \leq n-2k,\\[1mm]
I_2=\{a_2=-1, a_i, b_2, \beta, \delta_3, \theta_n\}, & 4 \leq i \leq n-2k,\\[1mm]
I_3=\{a_i, \beta\}, & 2\leq i \leq n-2k, \ i\neq 3, a_2\notin\{-1, 0\},\\[1mm]
I_4=\{a_1, a_2, a_{n-2k}, b_1, b_2, \delta_1, \delta_2, \delta_3, \theta_{n-k+i}\}, & 1 \leq i \leq k,\\[1mm]
\end{array}
\]
where for parameters from the set $I_4$ we have the relations

\begin{align*}
\delta_{1} & =-a_1a_{n-2k},  \\
b_{2} & =\frac{(-1)^{k-1}(a_2+1)^k\theta_{n}}{a_1},\\
\theta_{n-k+i} & =(-1)^{k-i}(a_2+1)^{k-i}\theta_{n}, \qquad 1\leq i\leq k-1.
\end{align*}

\begin{prop} \[R(\mu_3, 1)(I)=\bigcup_{j=1}^{4} R(\mu_3, 1)(I_j).\]
\end{prop}
\begin{proof} Thanks to Lemma \ref{lemmu3} for analysis of equalities \eqref{eq11} it is sufficient to consider the following cases.

{\bf Case 1.} Let $a_1=0$. Then we get $\delta_1=\theta_{n-k+i}=0, \ 1\leq i\leq k-1$ and
$a_2\gamma=(a_2+1)\theta_{n}=0$.

By applying the change $e_1^\prime=e_1+\sum\limits_{i=1}^{k}(-1)^{i}b_{1}f_{i}$, we obtain $b_{1}=0$.

{\bf Case 1.1.} Let $a_2=0$. Then $\theta_{n}=0$.

By taking the change of basis elements
\[e_1^\prime=e_1+\sum\limits_{i=1}^{k}(-1)^{i}b_{2}f_{k+i}, \qquad  x^\prime=x+\sum\limits_{i=1}^{k}(-1)^{i}\delta_{3}f_{k+i},\]
we can assume $b_{2}=\delta_{3}=0$. Thus, the subfamily of algebras $R(\mu_3, 1)(I_1)$ is obtained.

{\bf Case 1.2.} Let $a_2\neq0$. Then we have $\gamma=(a_2+1)\theta_{n}=0$. By setting $x^\prime=x-\frac{\delta_{2}}{a_2}e_{n-2k}$, we get $\delta_{2}=0$.

{\bf Case 1.2.1.} Let $a_2=-1$. Then we obtain the subfamily $R(\mu_3, 1)(I_2)$.

{\bf Case 1.2.2.} Let $a_2\neq-1$. Then $\theta_{n}=0$.

By putting
\[e_1^\prime=e_1+\sum\limits_{i=1}^{k}(-1)^{i}\frac{b_{2}}{(a_2+1)^i}f_{k+i}, \qquad x^\prime=x+\sum\limits_{i=1}^{k}(-1)^{i}\frac{\delta_{3}}{(a_2+1)^i}f_{k+i},\] we get $b_{2}=\delta_{3}=0$.

Thus, we have the class of algebras $R(\mu_3, 1)(I_3)$.

{\bf Case 2.} Let $a_1\neq0$. Then $\gamma=0$. By applying the change
\[e_1^\prime=e_1, \quad  e_i^\prime=e_i+\sum\limits_{j=i+1}^{n-2k}A_{j-i+2}e_{j},  \qquad 2\leq i\leq n-2k,\]
with
\[A_{3}=0, \quad a_3=0, \qquad
A_{i}=-\frac{1}{(i-2)a_1} \, \Big(\sum\limits_{j=3}^{i-1}A_ja_{i+2-j}+a_i\Big),
\ 4\leq i\leq n-2k-1,\]
\[ \qquad   \qquad  \quad A_{n-2k}=-\frac{1}{(n-2k-2)a_1}\,\Big(\sum\limits_{j=3}^{n-2k-1}A_ja_{n-2k+2-j}+
\beta\Big),\]
we derive $\beta=a_{i}=0, \ 4\leq i\leq n-2k-1$.

Therefore, we obtain the class of algebras $R(\mu_3, 1)(I_4)$.
\end{proof}

\begin{rem} It should be noted that algebras from different subfamilies $R(\mu_3, 1)(I_j), 1\leq j \leq 4$, are not isomorphic.
\end{rem}

\section*{ Acknowledgements}

This work was partially supported by  Ministerio de Econom\'ia y Competitividad (Spain),
grant MTM2013-43687-P (European FEDER support included), by Xunta de Galicia, grant GRC2013-045 (European FEDER support included)
 and by Kazakhstan Ministry of Education and Science, grant 0828/GF4: ``Algebras, close to Lie:
cohomologies, identities and deformations''.
The last named author was partially supported by a grant from the Simons Foundation.

\end{document}